\documentclass[reqno,a4paper,12pt]{amsart}

\usepackage[all,poly]{xy}
\usepackage{amsfonts,stmaryrd}
\usepackage[mathcal]{eucal}
\usepackage{amssymb}
\usepackage{amsmath}
\usepackage{mathrsfs}
\usepackage{enumerate}

\theoremstyle{plain}

\newtheorem {The}{Theorem}
\newtheorem{oldtheorem}{Theorem}

\newtheorem{lemma}{Lemma}
\newtheorem{problem}{Problem}

\def\map{\longrightarrow}

\def\GL{\operatorname{GL}}

\def\SL{\operatorname{SL}}

\def\EE{\operatorname{EE}}

\def\K{\operatorname{K}}

\def\unlhd{\trianglelefteq}

\def\map{\longrightarrow}

\def\map{\longrightarrow}

\def\epsilon{\varepsilon}

\def\pamod#1{\,(\operatorname{mod}{\, #1})\,}

\def\map{\longrightarrow}

\def\K{\operatorname{K}}

\def\SL{\operatorname{SL}}
\def\GL{\operatorname{GL}}

\long\def\forget#1\forgotten{}

\title[relative and unrelative elementary groups,
 revisited]{commutators of relative and unrelative\\ elementary groups,
 revisited}

\author{N.~Vavilov}
\address{Department of Mathematics and Mechanics,
St.~Petersburg State University, St.~Petersburg, Russia}
\email{nikolai-vavilov@yandex.ru}
\thanks{The work of the first author was supported by the
Russian Science Foundation grant 17-11-01261.}
\author{Z.~Zhang}
\address{Department of  Mathematics, Beijing Institute
of Technology, Beijing, China}
\email{zuhong@hotmail.com}


\begin{document}

\begin{abstract}
Let $R$ be any associative ring with $1$, $n\ge 3$, and let $A,B$ 
be two-sided ideals of $R$. In the present paper we show that 
the mixed commutator subgroup $[E(n,R,A),E(n,R,B)]$ is generated 
as a group by the elements of the two following forms: 1)
$z_{ij}(ab,c)$ and $z_{ij}(ba,c)$, 2) $[t_{ij}(a),t_{ji}(b)]$,
where $1\le i\neq j\le n$, $a\in A$, $b\in B$, $c\in R$. Moreover,
for the second type of generators, it suffices to fix one pair of
indices $(i,j)$. This result is both stronger and more general than the previous results by Roozbeh Hazrat and the authors. In particular,
it implies that for all associative rings one has the equality  
$\big[E(n,R,A),E(n,R,B)\big]=\big[E(n,A),E(n,B)\big]$ and many
further corollaries can be derived for rings subject to commutativity
conditions.
\end{abstract}

\maketitle

\maketitle
\hangindent 6cm\hangafter=0\noindent
To the remarkable St~Petersburg algebraist\\ Alexander Generalov
\par\hangindent 6cm\hangafter=0\noindent
\bigskip

\section{Introduction}

In the present note we generalize and strengthen the results 
by Roozbeh Hazrat and the authors \cite{Hazrat_Zhang_multiple, Hazrat_Vavilov_Zhang, NV18} on generation of mutual 
commutator subgroups of relative and unrelative elementary 
subgroups in the general linear group. Namely, we both 
dramatically reduce the sets of generators that occur therein 
and either seriously weaken, or completely remove commutativity conditions.

Let $R$ be an associative ring with 1, and $\GL(n, R)$ be the general linear group of degree $n\ge 3$ over $R$. As
usual, $e$ denotes the identity matrix, whereas $e_{ij}$ denotes a standard matrix unit. For $c\in R$ and $1\le i\neq j\le n$,
we denote by $t_{ij}(c) = e +ce_{ij}$ the corresponding elementary transvection. To an ideal $A\unlhd R$, we assign the
elementary subgroup
$$ E(n,A)=\big\langle t_{ij}(a),\ a\in A,\ 1\le i\neq j\le n\big\rangle. $$
\par
The corresponding relative elementary subgroup $E(n,R,A)$ is defined as the normal closure of $E(n,A)$ in the absolute elementary subgroup $E(n,R)$. From the work of Michael Stein, Jacques Tits, and Leonid
Vaserstein it is classically known that {\it as a group\/} $E(n,R,A)$
is generated by $z_{ij}(a,c)=t_{ji}(c)t_{ij}(a)t_{ji}(-c)$, where
$1\le i\neq j\le n$, $a\in A$, $c\in R$.

Further, consider the reduction homomorphism $\rho_I:\GL(n,R)\map\GL(n,R/I)$ modulo $I$. By definition, the principal congruence subgroup $\GL(n,I)=\GL(n,R,I)$ is the kernel of $\rho_I$. In other words, $\GL(n,I)$ consists of all matrices $g$ congruent to $e$ 
modulo $I$.

A first version of following result was discovered (in a slightly
less precise form) by Roozbeh Hazrat and the second author, see \cite{Hazrat_Zhang_multiple}, Lemma 12. In exactly this form 
it is stated in our paper \cite{Hazrat_Vavilov_Zhang}, Theorem~3A.

\begin{oldtheorem}
Let $R$ be a quasi-finite ring with $1$, let $n\ge 3$, and let $A,B$ 
be two-sided ideals of $R$. Then the mixed commutator subgroup 
$[E(n,R,A),E(n,R,B)]$ is generated as a group by the elements of the form
\par\smallskip
$\bullet$ $z_{ij}(ab,c)$ and $z_{ij}(ba,c)$,
\par\smallskip
$\bullet$ $[t_{ij}(a),t_{ji}(b)]$,
\par\smallskip
$\bullet$ $[t_{ij}(a),z_{ij}(b,c)]$,
\par\smallskip\noindent
where $1\le i\neq j\le n$, $a\in A$, $b\in B$, $c\in R$.
\end{oldtheorem}

In the present paper, we prove the following result, which
is both {\it terribly\/} much stronger, and much more general 
than Theorem~A and which completely solves \cite{Hazrat_Vavilov_Zhang}, Problem 1, for 
the case of $\GL_n$.

\begin{The}
Let $R$ be any associative ring with $1$, let $n\ge 3$, and let $A,B$ 
be two-sided ideals of $R$. Then the mixed commutator subgroup 
$[E(n,R,A),E(n,R,B)]$ is generated as a group by the elements of the form
\par\smallskip
$\bullet$ $z_{ij}(ab,c)$ and $z_{ij}(ba,c)$,
\par\smallskip
$\bullet$ $[t_{ij}(a),t_{ji}(b)]$,
\par\smallskip\noindent
where $1\le i\neq j\le n$, $a\in A$, $b\in B$, $c\in R$. Moreover,
for the second type of generators, it suffices to fix one pair of
indices $(i,j)$.
\end{The}

Let us briefly review the sequence of events that led us to 
this result. It all started a decade ago with our joint papers
with Alexei Stepanov and Roozbeh Hazrat  \cite{Vavilov_Stepanov_standard, Hazrat_Zhang, Vavilov_Stepanov_revisited}, which, in particular, gave {\it  three\/} 
completely different proofs of the following birelative standard commutator formula, under various commutativity conditions\footnote{The third of these proofs was essentially reduction to the absolute case via level calculations, as discovered earlier by Hong You \cite{HongYou}, of which we were not aware at the time of writing
these papers.}. 
In turn, this formula generalised a great 
number of preceding results due to Hyman Bass, Alec Mason and
Wilson Stothers, Andrei Suslin, Leonid Vaserstein, Zenon Borewicz 
and the first author, and many others, see, for instance
\cite{Bass_stable, Mason_Stothers, Mason_commutators_2,
Suslin, Vaserstein_normal,borvav}, and a complete bibliography
of early papers in 
\cite{Hahn_OMeara, Vavilov_Stepanov_1, Hazrat_Vavilov_Bak65}.
Compare also \cite{yoga-1, Porto-Cesareo, yoga-2, Hazrat_Vavilov_Zhang} for a detailed description of the recent work
in the area.

\begin{oldtheorem}
Let $R$ be a quasi-finite ring with $1$, let $n\ge 3$, and let $A,B$ 
be two-sided ideals of $R$. Then the following birelative standard commutator formula holds
$$ \big[E(n,R,A),\GL(n, R, B)\big]=\big[E(n,R,A),E(n,R,B)\big]. $$
\end{oldtheorem}

The condition in the above theorem is very general and embraces 
very broad class of rings, but some commutativity condition is 
necessary here, since it is known that the standard commutator 
formula may fail for general associative rings even in the absolute 
case, see \cite{Gerasimov}.

Last year the first author noticed that for {\it commutative\/}
rings an argument from his paper with Alexei Stepanov \cite{Stepanov_Vavilov_decomposition} implies that the standard commutator formula holds also in the following unrelativised 
form, see \cite{NV18}, Theorem~1.

\begin{oldtheorem}
Let $R$ be a commutative ring with $1$, let $n\ge 3$, and let $A,B$ 
be two-sided ideals of $R$. Then the following commutator formula 
holds
$$ \big[E(n,A),\GL(n, R, B)\big]=\big[E(n,A),E(n,B)\big]. $$
\end{oldtheorem}

In particular, this immediately implies the following striking equality, 
see \cite{NV18}, Theorem~2.

\begin{oldtheorem}
Let $R$ be a commutative ring with $1$, let $n\ge 3$, and let $A,B$ 
be two-sided ideals of $R$. Then one has
$$ \big[E(n,R,A),E(n,R,B)\big]=\big[E(n,A),E(n,B)\big]. $$
\end{oldtheorem}

Thereupon, the second author immediately suggested that since
everything occurs inside the absolute elementary group $E(n,R)$,
one should be able to prove Theorem D directly, by looking at the elementary generators in Theorem A and proving that the third
type of generators are redundant. Over {\it commutative\/} rings
this was essentially accomplished in the more general context of
Chevalley groups in our joint paper \cite{NZ1}.

Immediately thereafter we started to work on the unitary sequel
\cite{NZ2} and discovered that most of the requisite results, 
apart from the unitary analogue of Theorem A, hold for {\it arbitrary\/}
form rings, without any commutativity conditions. This propted
us to look closer inside the proofs of \cite{Hazrat_Vavilov_Zhang}, Theorems 3A and 3B, and the Main Lemmas of \cite{NZ1,NZ2}.
We discovered that all references to Theorem~B or any of its
special cases can be easily replaced by elementary calculations
that hold over arbitrary associative rings, and only depend on
Steinberg relations (so that they can be carried out already in 
Steinberg groups).

Finally, attempting to strengthen the birelative standard commutator formula in the arithmetic case \cite{NV19}, the first author was 
forced to look
for a stronger version of Theorem 1, with demoted set of generators. 
But a calculation that procures such a reduction was already 
contained in the papers on
bounded generation, see, for instance, \cite{CK83, Tavgen1990}.
A similar calculation is hidden also in the proof of the main theorem 
in the recent preprint by Andrei Lavrenov and Sergei Sinchuk 
\cite{Lavrenov_Sinchuk}.
Observe that, as discovered by Wilberd van der Kallen 
\cite{vdK-group}, and amply developed by Stepanov \cite{Stepanov_calculus, Stepanov_nonabelian}, one could 
reduce also the number of requisite $z_{ij}(a,c)$'s.

Since both types of generators in Theorem 1 already belong to
$\big[E(n,A),E(n,B)\big]$, we get the following generalisation of
Theorem D, for {\it arbitrary\/} associative rings.

\begin{The}
Let $R$ be any associative ring with $1$, let $n\ge 3$, and let $A,B$ 
be two-sided ideals of $R$.  Then one has
$$ \big[E(n,R,A),E(n,R,B)\big]=\big[E(n,A),E(n,B)\big]. $$
\end{The}

In turn, together with Theorem B this last result immediately
implies the following very broad generalisation of Theorem C.

\begin{The}
Let $R$ be a quasi-finite ring with $1$, let $n\ge 3$, and let $A,B$ 
be two-sided ideals of $R$. Then the following commutator formula 
holds
$$ \big[E(n,A),\GL(n, R, B)\big]=\big[E(n,A),E(n,B)\big]. $$
\end{The}

In \S~2 we prove Theorem 1, and thus also Theorems~2
and~3. Finally, in \S~3 we establish some further corollaries and
variations of these results and make some further related 
observations.

In the present paper we describe part of the astounding recent 
progress in the direction of unrelativisation, whose first steps
were presented in our talk ``Relativisation and unrelativisation'' 
at the Polynomial Computer Algebra 2019 (see
http://pca-pdmi.ru/2019/program, April 19).


\section{The proof of Theorem 1}

Everywhere below the commutators are left-normed 
so that for two elements $x,y$ of a group $G$ one has 
$[x,y]={x}^y\cdot y^{-1}=xyx^{-1}y^{-1}$. In the sequel
we repeatedly use standard commutator identities such as 
$[xy,z]={}^x[y,z]\cdot [x,z]$ or $[x,yz]=[x,y]\cdot{}^y[x,z]$ 
without any explicit reference.

The following lemma is a classical result due to Stein, Tits and 
Vaserstein, see \cite{Vaserstein_normal}. 

\begin{lemma}
Let $R$ be an associative ring with $1$, $n\ge 3$, and let $A$ 
be a two-sided ideal of $R$. Then as a subgroup $E(n,R,A)$ is 
generated by $z_{ij}(a,c)$, for all $1\le i\neq j\le n$, $a\in A$, 
$c\in R$.
\end{lemma}

Since $E(n,R,A)$ is normal in $E(n,R)$ by the very definition, in
particular this lemma implies that every elementary conjugate 
of  $z_{ij}(a,c)$ is again a product of generators of the same type.

The following result is \cite{Hazrat_Zhang_multiple}, Lemma 12,
a detailed proof is also reproduced in \cite{Hazrat_Vavilov_Zhang},
Lemma 2A.
\begin{lemma}
Let $R$ be an associative ring with $1$, let $n\ge 3$, and let $A,B$ 
be two-sided ideals of $R$. Then the mixed commutator subgroup 
$[E(n,R,A),E(n,R,B)]$ is generated as a group by the elements of the form
\par\smallskip
$\bullet$ ${}^x z_{ij}(ab,c)$ and ${}^x z_{ij}(ba,c)$,
\par\smallskip
$\bullet$ ${}^x [t_{ij}(a),t_{ji}(b)]$,
\par\smallskip
$\bullet$ ${}^x [t_{ij}(a),z_{ij}(b,c)]$,
\par\smallskip\noindent
where $1\le i\neq j\le n$, $a\in A$, $b\in B$, $c\in R$, and 
$x\in E(n,R)$.
\end{lemma}

Both types of generators in the first item belong to $E(n,R,AB+BA)$
and Lemma~1 implies that for them $x$ can be removed, without
affecting the subgroup they generate. The first type of generators 
listed in Theorem 1 still generate the same group $E(n,R,AB+BA)$.
\par
An obvious level calculation (see, for instance, \cite{Vavilov_Stepanov_standard}, Lemma 3 or \cite{Hazrat_Vavilov_Zhang}, Lemma~1A) shows the other two 
types of generators listed in Theorem A still belong to the congruence 
subgroup $\GL(n,R,AB+BA)$. Now, in the conditions of Theorem B,
an elementary conjugate of these generators is again the same
generator, modulo the subgroup $E(n,R,AB+BA)$.
\par
However, it is very easy to get rid of any reference to Theorem B
here. Let us start with the second type of generators 
$z=[t_{ij}(a),t_{ji}(b)]$. 

\begin{lemma}
Let $R$ be an associative ring with $1$, $n\ge 3$, and let $A,B$ 
be two-sided ideals of $R$. Then for any  $1\le i\neq j\le n$, 
$a\in A$, $b\in B$, and any $x\in E(n,R)$ the conjugate
${}^x [t_{ij}(a),t_{ji}(b)]$ is congruent to $[t_{ij}(a),t_{ji}(b)]$
modulo $E(n,R,AB+BA)$.
\end{lemma}
\begin{proof}
Clearly, $z=[t_{ij}(a),t_{ji}(b)]$ resides in the image of
the fundamental embedding of $E(2,R)$ into $E(n,R)$ in the
$i$-th and $j$-th rows and columns, where one has
\par\smallskip
$$ z=\left[\begin{pmatrix} 1&a\\ 0&1\end{pmatrix},
\begin{pmatrix} 1&0\\ b&1\end{pmatrix}\right]=
\begin{pmatrix} 1+ab+abab&-aba\\ bab&1-ba\end{pmatrix} $$
\noindent
and 
$$ z^{-1}=
\left[\begin{pmatrix} 1&0\\ b&1\end{pmatrix},
\begin{pmatrix} 1&a\\ 0&1\end{pmatrix}\right]=
\begin{pmatrix} 1-ab&aba\\ -bab&1+ba+baba\end{pmatrix}. $$
\par\smallskip
Consider the elementary conjugate ${}^xz$. We argue by induction 
on the length of $x\in E(n,R)$ in elementary generators. Let
$x=yt_{kl}(c)$, where $y\in E(n,R)$ is shorter than $x$, whereas
$1\le k\neq l\le n$, $c\in R$. 
\par\smallskip
$\bullet$ If $k,l\neq i,j$, then $t_{kl}(c)$ commutes with $z$ 
and can be discarded.
\par\smallskip
$\bullet$ On the other hand, for any $h\neq i,j$ the above formulas 
for $z$ and $z^{-1}$ immediately imply that
\begin{alignat*}{2}
&[t_{ih}(c),z]=t_{ih}(-abc-ababc)t_{jh}(-babc),\quad
&&[t_{jh}(c),z]=t_{ih}(abac)t_{jh}(bac),\\
&[t_{hi}(c),z]=t_{hi}(cab)t_{hj}(-caba),
&&[t_{hj}(c),z]=t_{hi}(cbab)t_{hj}(-cba-cbaba).
\end{alignat*}
\par\noindent
All factors on the right hand side belong already to $E(n,AB+BA)$
This means that
$$ {}^xz\equiv {}^yz \pamod{E(n,R,AB+BA)}. $$
\par\smallskip
$\bullet$ Finally, for $(k,l)=(i,j),(j,i)$ we can take an $h\neq i,j$
and rewrite $t_{kl}(c)$ as a commutator $t_{ij}(c)=[t_{ih}(c),t_{hj}(1)]$
or $t_{ji}(c)=[t_{h}(c),t_{hi}(1)]$ and apply the previous item to
get the same congruence modulo $E(n,R,AB+BA)$.
\par\smallskip
By induction we get that ${}^xz\equiv z\pamod{E(n,R,AB+BA)}$.
\end{proof}

Thus, to prove the first claim of Theorem 1 it only remains to 
establish the following lemma. For commutative rings it is essentially
the simplest special case of the Main Lemma of \cite{NZ1}. However,
there it is expressed in the language of root elements. Even though 
commutativity is not used in the proof in any material way, it is
formally assumed. For completeness, we reproduce the proof of a
somewhat stronger fact, in matrix notation.

\begin{lemma}
Let $R$ be an associative ring with $1$, $n\ge 3$, and let $A,B$ 
be two-sided ideals of $R$. Then for any  $1\le i\neq j\le n$, 
$a\in A$, $b\in B$, $c\in R$
and any $x\in E(n,R)$ the conjugate ${}^x[t_{ij}(a),z_{ij}(b,c)]$ 
of a generator of the third type 
is congruent to an elementary conjugate of some generator 
$[t_{kl}(a'),t_{lk}(b')]$, $1\le k\neq l\le n$, $a'\in A$, $b'\in B$,
of the second type, modulo $E(n,R,AB+BA)$.
\end{lemma}
\begin{proof}
Indeed, let $z=[t_{ij}(a),z_{ij}(b,c)]$. Take any $h\neq i,j$. Then
$$ z=[t_{ij}(a),z_{ij}(b,c)]=t_{ij}(a)\cdot{}^{z_{ij}(b,c)}t_{ij}(-a)= 
t_{ij}(a)\cdot{}^{z_{ij}(b,c)}[t_{ih}(1),t_{hj}(-a)]. $$
\noindent
Thus, 
\begin{multline*}
z=t_{ij}(a)\cdot[{}^{z_{ij}(b,c)}t_{ih}(1),{}^{z_{ij}(b,c)}t_{hj}(-a)]=\\
t_{ij}(a)\cdot [t_{ih}(1-bc)t_{jh}(-cbc),t_{hi}(-acbc)t_{hj}(-a(1-cb))]=\\
t_{ij}(a)\cdot[t_{ih}(1)u,t_{hj}(-a)v],
\end{multline*}
\noindent
where
$$ u=t_{jh}(-cbc)t_{ih}(-bc)\in E(n,B),\quad 
v=t_{hi}(-acbc)t_{hj}(acb)\in E(n,AB). $$
\noindent
Thus, 
$$ z\equiv t_{ij}(a)\cdot[t_{ih}(1)u,t_{hj}(-a)] \pamod{E(n,R,AB+BA)}. $$
\noindent
On the other hand, 
$$  t_{ij}(a)\cdot [t_{ih}(1)u,t_{hj}(-a)]=
 t_{ij}(a)\cdot{}^{t_{ih}(1)}[u,t_{hj}(-a)]\cdot 
t_{ij}(-a), $$
\par\noindent
whereas
\begin{multline*}
[u,t_{hj}(-a)]={}^{t_{jh}(-cbc)}t_{ih}(bca)\cdot 
[t_{jh}(-cbc),t_{hj}(-a)]\equiv \\ 
[t_{jh}(-cbc),t_{hj}(-a)]  \pamod{E(n,R,AB+BA)}. 
 \end{multline*}
 \par
 Summarising the above, we see that
 $$ {}^xz\equiv {}^{xt_{ij}(a)t_{ih}(1)}[t_{jh}(-cbc),t_{hj}(-a)]  \pamod{E(n,R,AB+BA)}, $$
\par\noindent
where $[t_{jh}(-cbc),t_{hj}(-a)]$ is the second type generator,
as claimed. 
\end{proof}

At this point we have already established the first claim of 
Theorem 1 --- and thus also Theorems 2 and 3. The rest is 
a bonus, that we need for more sophisticated applications.
The proof of the final claim of Theorem 1 is in fact a refinement 
of the proof of \cite{NV19}, Theorem 3. Again,
formally commutativity was assumed there, but can be easily
circumvented.

\begin{lemma}
Let $R$ be an associative ring with $1$, $n\ge 3$, and let $A,B$ 
be two-sided ideals of $R$. Then for any  $1\le i\neq j\le n$, any
$1\le k\neq l\le n$, and $a\in A$, $b\in B$, the elementary
commutator
$[t_{ij}(a),t_{ji}(b)]$ is congruent to $[t_{kl}(a),t_{lk}(b)]$
modulo $E(n,R,AB+BA)$.
\end{lemma}
\begin{proof}
Take any $ h\neq i,j$ and rewrite the elementary commutator
$z=\big[t_{ij}(a),t_{ji}(b)\big]$ as
$$ z=t_{ij}(a)\cdot{}^{t_{ji}(b)}t_{ij}(-a)=
t_{ij}(a)\cdot{}^{t_{ji}(b)}\big[t_{ih}(a),t_{hj}(-1)\big]. $$
\noindent
Expanding the conjugation by $t_{ji}(b)$, we see that 
$$ z=t_{ij}(a)\cdot\big[{}^{t_{ji}(b)}t_{ih}(a),{}^{t_{ji}(b)}t_{hj}(-1)\big] = t_{ij}(a)\cdot[t_{jh}(ba)t_{ih}(a),t_{hj}(-1)t_{hi}(b)\big]. $$
\noindent
Now, the first factor $t_{jh}(ba)$ of the first argument in this last commutator already belongs to the group $E(n,BA)$ which is 
contained in $E(n,R,AB+BA)$. Thus, as above, 
$$ z\equiv  t_{ij}(a)\cdot[t_{ih}(a),t_{hj}(-1)t_{hi}(b)\big] \pamod{E(n,R,AB+BA)}. $$
\noindent
Using multiplicativity of the commutator w.r.t. the second argument, cancelling the first two factors of the resulting 
expression, and then applying Lemma~3 we see that
$$ z\equiv 
{}^{t_{hj}(-1)}\big[t_{ih}(a),t_{hi}(b)\big] 
\equiv \big[t_{ih}(a),t_{hi}(b)\big] 
\pamod{E(n,R,AB+BA)}. $$
\par
Similarly, rewriting the commutator $z$ differently, as
$$ z=\big[t_{ij}(a),t_{ji}(b)\big]={}^{t_{ij}(a)}t_{ji}(b)\cdot t_{ji}(-b)=
{}^{t_{ij}(a)}\big[t_{jh}(b),t_{hi}(1)\big]\cdot t_{ji}(-b), $$
\noindent
we get the congruence 
$$ z\equiv \big[t_{hj}(a),t_{jh}(b)\big] 
\pamod{E(n,R,AB+BA)}. $$
\par
Obviously, for $n\ge 3$ we can pass from any position $(i,j)$, 
$i\neq j$, to any other such position $(k,l)$, $k\neq l$, by a 
sequence of at most three such elementary moves.
\end{proof}

This finishes the proof of Theorem 1.


\section{Further variations and final remarks}

The following result is a generalisation of the unrelative normality theorem by Bogdan Nica, see \cite{Nica}, Theorem 2, which
pertained to the commutative case. It is an immediate corollary of
our Theorem 3.

\begin{The}
Let $R$ be a quasi-finite ring with $1$, let $n\ge 3$, and let $A$ 
be a two sided ideal of $R$. Then $E(n,A)$ is normal in 
$\GL(n, R, A)$.
\end{The}

Let us mention another amazing corollary of Theorem 1, in the
style of stability results without stability conditions by Tony Bak,
see \cite{Bak}. With this end, observe that Lemma~3 implies that the quotient
$$ \big[E(n,A),E(n,B)\big]/E(n,R,AB+BA) $$ 
\noindent
is central in $E(n,R)/E(n,R,AB+BA)$. In other words, the
following holds.

\begin{lemma}
Let $R$ be an associative ring with $1$, $n\ge 3$, and let $A,B$ 
be two-sided ideals of $R$. Then
$$ \big[\big[E(n,A),E(n,B)\big],E(n,R)\big]=E(n,R,AB+BA). $$
\end{lemma}

But now Theorem 1 implies surjective stability of such quotients,
which is a generalisation of the first half of \cite{Hazrat_Vavilov_Zhang}, Lemma 15, to arbitrary associative
rings, without any stability conditions, or commutativity conditions.
Indeed, in view of Theorem 1 and Lemma 6 as a normal subgroup of $E(n,R)$ the group $[E(n,A),E(n,B)]$ is generated by 
$[E(3,A),E(3,B)]$. This can be restated as follows.

\begin{The}
Let $R$ be any associative ring with $1$, and let $A$ 
and $B$ be two sided ideals of $R$. Then for all $n\ge 3$ the
stability map 
\begin{multline*}
\big[E(n,A),E(n,B)\big]/E(n,R,AB+BA) \map\\
\big[E(n+1,A),E(n+1,B)\big]/E(n+1,R,AB+BA) 
\end{multline*}
\noindent
is surjective.
\end{The}

This quotient occurs surprisingly often in seemingly unrelated
problems, and is so interesting in itself, that we are now hatching
the idea to definitively comprehend its structure.
Lemmas 3 and 5 assert that some elementary commutators
and their conjugates are congruent modulo $E(n,R,AB+BA)$.
We have several further results in the same spirit. For instance,
one has
\par\smallskip
$\bullet$
$\big[t_{ij}(ac),t_{ji}(b)\big] \equiv \big[t_{ij}(a),t_{ji}(cb)\big] 
\pamod{E(n,R,AB+BA)}$,
\par\smallskip
$\bullet$
$\big[t_{ij}(a_1+a_2),t_{ji}(b)\big] \equiv 
\big[t_{ij}(a_1),t_{ji}(b)\big]\cdot\big[t_{ij}(a_2),t_{ji}(b)\big] 
\pamod{E(n,R,AB+BA)}$,
\par\smallskip
$\bullet$ $\big[t_{ij}(a),t_{ji}(b_1+b_2)\big] \equiv 
\big[t_{ij}(a),t_{ji}(b_1)\big]\cdot\big[t_{ij}(a),t_{ji}(b_2)\big] 
\pamod{E(n,R,AB+BA)}$,
\par\smallskip
$\bullet$ ${\big[t_{ij}(a),t_{ji}(b)\big]}^{-1} 
\equiv \big[t_{ij}(-a),t_{ji}(b)\big] 
\equiv \big[t_{ij}(a),t_{ji}(-b)\big] 
\pamod{E(n,R,AB+BA)}$,
\par\smallskip
$\bullet$ $\big[t_{ij}(a_1),t_{ji}(b)\big]\equiv 
\big[t_{ij}(a_2),t_{ji}(b)\big] \pamod{E(n,R,AB+BA)}$,\par
{}\hskip3truein if $a_1\equiv a_2\pamod{AB+BA+A^2}$
\par\smallskip
$\bullet$ $\big[t_{ij}(a),t_{ji}(b_1)\big]\equiv 
\big[t_{ij}(a),t_{ji}(b_1)\big] \pamod{E(n,R,AB+BA)}$,\par
{}\hskip3truein if $b_1\equiv b_2\pamod{AB+BA+B^2}$
\par\smallskip
\noindent
etc. We have not made any attempt to systematically collect all
such congruences in the present article, since they are not
directly needed to prove Theorem 1. But they may turn out
very useful to control the quotient 
$\big[E(n,A),E(n,B)\big]/E(n,R,AB+BA)$. Observe that by Lemma 6
this quotient is central in $E(n,R)/E(n,R,AB+BA)$, and thus,
in particular, it is itself abelian. We intend to list all such
properties in a subsequent paper, where we propose to assail
the following tantalising problem.

\begin{problem}
Give a presentation of
$$ \big[E(n,A),E(n,B)\big]/\EE(n,R,AB+BA) $$
\noindent
by generators and relations.
\end{problem}

Let us mention yet another corollary of Theorem 1. Let $U(n,R)$ 
and $U^-(n,R)$ be the groups of upper unitriangular 
and lower unitriangular matrices, respectively. These are unipotent 
radicals of the standard Borel subgroup, and its opposite Borel 
subgroup. Further, set
$$ U(n,I)=U(n,R)\cap \GL(n,R,I),\quad 
U^-(n,I)=U^-(n,R)\cap \GL(n,R,I). $$
\noindent
In \cite{NV18} we considered another birelative group
$$ \EE(n,A,B)=\big\langle U(n,A), U^-(n,B)\big\rangle, $$
\noindent 
and established that for {\it commutative\/} rings this group
contains $\big[E(n,A),E(n,B)\big]$, see \cite{NV18}, Theorem 3. 
Since in the case $\EE(n,A,B)$ contains $E(n,R,AB)$ by
\cite{NV19}, Lemma 8, this theorem immediately follows 
from our Theorem 1. It is natural to expect that an analogue
of this result holds over arbitrary associative rings.

\begin{problem}
Let $R$ be any associative ring with $1$, let $n\ge 3$, and let $A,B$ 
be two-sided ideals of $R$. Prove that 
$$ \big[E(n,A),E(n,B)\big]\le\EE(n,A,B). $$
\end{problem}

The difficulty now is exactly to prove that $E(n,R,AB+BA)\le\EE(n,A,B)$.
In the non-commutative case the argument used in the proof
of \cite{NV19}, Lemma 8, only works in the following form.

If $i,j\neq n$, there exists an $h>i,j$ so that one can express 
$t_{ij}(ab)$ as
$$ t_{ij}(ab)=\big[t_{ih}(a),t_{hj}(b)\big]\in
\big[U(n,A), U^-(n,B)\big], $$
\noindent
and conclude that $z_{ij}(ab,c)\in\EE(n,A,B)$.

Similarly, $i,j\neq 1$, there exists an $h<i,j$ so that one can 
express $t_{ij}(ba)$ as
$$ t_{ij}(ba)=\big[t_{ih}(b),t_{hj}(a)\big]\in
\big[U^-(n,B),U(n,A)\big], $$
\noindent
and conclude that $z_{ij}(ba,c)\in\EE(n,A,B)$.

In the case of commutative rings (or, more generally,
when $AB=BA$) this implies that $E(n,R,AB+BA)\in\EE(n,A,B)$,
see \cite{vdK-group, Stepanov_calculus, Stepanov_nonabelian}.
But in the general case this would require some additional 
reasoning.

Let us mention an even more challenging related question.
Namely, let $P$ be a proper
standard parabolic subgroup of $\GL(n,R)$. We can define the 
corresponding subgroup of $\EE(n,A,B)$ as follows:
$$ \EE_P(n,A,B)=\big\langle U_P(A), U_P^-(B)\big\rangle, $$
\noindent
where $U_P(A)$ and $U^-_P(B)$ are the intersections of $U(n,A)$
and $U^-(n,B)$ with the unipotent radicals $U_P$ and $U_P^-$ of
$P$ and its opposite standard parabolic $P^-$, respectively. In the
definition of $\EE(n,A,B)$ itself $P=B(n,R)$ is the standard Borel
subgroup. However, in many cases it is technically much more
expedient to work with the maximal standard parabolics instead,
see, for instance, the works by Alexei Stepanov 
\cite{Stepanov_calculus, Stepanov_nonabelian}.

\begin{problem}
Let $R$ be any associative ring with $1$, let $n\ge 3$, and let $A,B$ 
be two-sided ideals of $R$. Prove that 
$$ \big[E(n,A),E(n,B)\big]\le\EE_P(n,A,B). $$
\end{problem}

The next problem proposes to generalise
\cite{yoga-2}, Theorem 8A, and \cite{Hazrat_Vavilov_Zhang}, 
Theorem 5A, from quasi-finite rings, to arbitrary associative rings. 
In other words, to prove that any multiple commutator of relative or
unrelative elementary subgroups is  equal to some double such
commutator, see  \cite{Hazrat_Zhang_multiple, yoga-2, RNZ5, Hazrat_Vavilov_Zhang, Stepanov_universal} 

Here $A\circ B=AB+BA$ stands for the symmetrised product of two sided ideals $A$ and $B$. In general, the symmetrised product is 
not associative. Thus, when writing something like $A\circ B\circ C$, 
we have to specify the order in which products are formed.
for notation pertaining
to multiple commutators.

Let $G$ be a group and $H_1,\ldots,H_m\le G$ be its
subgroups. There are many ways to form a higher commutator of these
groups, depending on where we put the brackets. Thus, for three
subgroups $F,H,K\le G$ one can form two triple commutators
$[[F,H],K]$ and $[F,[H,K]]$. Usually, we write $[H_1,H_2,\ldots,H_m]$ for the {\it left-normed\/} commutator, defined inductively by
$$ [H_1,\ldots,H_{m-1},H_m]=[[H_1,\ldots,H_{m-1}],H_m]. $$
\noindent
To stress that here we consider {\it any\/} commutator of these subgroups, with an arbitrary placement of brackets, we write $[\![H_1,H_2,\ldots,H_m]\!]$. Thus, for instance, $[\![F,H,K]\!]$ 
refers to any of the two arrangements above.
\par
Actually, a specific arrangment of brackets usually does not play
major role in our results -- apart from one important attribute. 
Namely, what will matter a lot is the position of the outermost 
pairs of inner brackets. Namely, every higher commutator subgroup
$[\![H_1,H_2,\ldots,H_m]\!]$ can be uniquely written as
$$ [\![H_1,H_2,\ldots,H_m]\!]=
[[\![H_1,\ldots,H_h]\!],[\![H_{h+1},\ldots,H_m]\!]], $$
\noindent
for some $h=1,\ldots,m-1$. This $h$ will be called the cut point
of our multiple commutator. 

\begin{problem}
Let $R$ be any associative ring with $1$, let $n\ge 3$, and let
$A_i\unlhd R$, $i=1,\ldots,m$,  be two-sided ideals of $R$. 
Consider an arbitrary arrangment of brackets\/ $[\![\ldots]\!]$
with the cut point\/ $h$. Then one has
$$
[\![E(n,I_1),E(n,I_2),\ldots,E(n,I_m)]\!]=
[E(n,I_1\circ\ldots\circ I_h),E(n,I_{h+1}\circ\ldots\circ I_m)],
$$
\noindent
where the bracketing of symmetrised products on the right hand side coincides with the bracketing of the commutators on the left hand side.
\end{problem}

Observe that Theorem~A and its analogues were used by 
Alexei Stepanov in his remarkable results on bounded 
width of commutators with respect to elementary generators, 
see \cite{Stepanov_universal}, and our survey \cite{Porto-Cesareo}.
Now, it would be natural to refer in these results to our new 
reduced set of generators from Theorem 1. 

Analogues of our Theorems 1 and 2 hold for Bak's unitary groups
over {\it arbitrary\/} form rings. In particular, this generalises 
\cite{RNZ5}, Theorem 9 and \cite{Hazrat_Vavilov_Zhang}, Theorem 3B.
Also, it solves \cite{Hazrat_Vavilov_Zhang}, Problem 1 for the 
unitary case. These results are now incorporated in our unitary paper
\cite{NZ2}. A full analogue of Theorem~1 for Chevalley groups
is much more difficult even in the commutative case, and will be published in \cite{NZ4}.

The authors thank Roozbeh Hazrat and Alexei Stepanov for ongoing 
discussion of this circle of ideas, and long-standing cooperation 
over the last decades. Also, we are grateful to Pavel Gvozdevsky 
and Sergei Sinchuk for their extremely pertinent questions during 
the seminar talk, where the first author was reporting \cite{NV19}.
Last, but not least, the first author thanks Nikolai Vasiliev 
for his insistence and friendly encouragement.


\end{document}